\documentclass[11pt,a4paper]{article}
\usepackage{lipsum}
\usepackage{authblk}
\usepackage{amsmath,a4}  
\usepackage[utf8]{inputenc}
\usepackage{graphicx,enumerate,color,amsfonts,amsthm,amsmath,amssymb,mathrsfs,marginnote,dsfont,textcomp,tikz,marginnote,subfigure,pdflscape,multirow,booktabs,float,enumitem,scrtime,currfile,titleps,hyphenat,algorithm,algpseudocode,bigints,diagbox,multirow,mathtools}
\usepackage[colorlinks,
    linkcolor={red!60!black},
    citecolor={black},
    urlcolor={blue!80!black}]{hyperref}
\usepackage{tikz,pgf,pgffor,pgfplots}
\tikzset{>=latex}
\usetikzlibrary{calc,patterns,decorations,intersections,decorations.pathmorphing,fit,backgrounds,arrows,shapes,snakes,automata,petri,positioning,,graphs,graphs.standard}
\usepgfmodule{shapes,plot}

\newtheorem{theorem}{Theorem}[section]

\newtheorem{lemma}[theorem]{Lemma}
\newtheorem{proposition}[theorem]{Proposition}

\theoremstyle{definition}

\theoremstyle{remark}

\makeatletter
\newcommand{\setword}[2]{%
  \phantomsection
  #1\def\@currentlabel{\unexpanded{#1}}\label{#2}%
}
\makeatother

\newcommand{\id}{\textnormal{\textbf{id}}}

\newcommand{\diag}{\textnormal{\textbf{diag}}}
\newcommand{\dt}{\textnormal{\textbf{det}}}
\newcommand{\rk}{\textnormal{\textbf{rk}}}

\newcommand{\tr}{\textnormal{\textbf{tr}}}

\newcommand{\R}{\mathbb{R}}
\newcommand{\RM}{\mathcal{R}}
\newcommand{\SB}{\textnormal{\textbf{S}}}
\newcommand{\IB}{\textnormal{\textbf{I}}}
\newcommand{\JB}{\textnormal{\textbf{J}}}
\newcommand{\RB}{\textnormal{\textbf{R}}}
\newcommand{\NB}{\textnormal{\textbf{N}}}
\newcommand{\XB}{\textnormal{\textbf{X}}}

\newcommand{\DFE}{\textnormal{\textbf{DFE}}}
\newcommand{\EE}{\textnormal{\textbf{EE}}}
\allowdisplaybreaks
\usepackage[top=2cm, bottom=2cm, left=2cm, right=2cm]{geometry}
\usepackage{fancyhdr}
\newpagestyle{main}{
\setheadrule{.4pt}
\sethead{}
        {}
        {\sectiontitle}
}
\renewenvironment{abstract}{%
\hfill\begin{minipage}{0.95\textwidth}
\rule{\textwidth}{1pt}}
{\par\noindent\rule{\textwidth}{1pt}\end{minipage}}
\makeatletter
\renewcommand\@maketitle{%
\hfill
\begin{minipage}{0.95\textwidth}
\vskip 2em
\let\footnote\thanks 
{\Large \bf \@title \par }
\vskip 1.5em
{\large \@author \par}
\end{minipage}
\vskip 1em \par
}
\makeatother
\begin{document}
%
\title{\nohyphens{Analysis of a network--based SIR model}}
\author[a]{Karunia Putra Wijaya}
\author[b,$\ast$]{Dipo Aldila}
\affil[a]{\small\emph{Mathematical Institute, University of Koblenz, 56070 Koblenz, Germany}}
\affil[b]{\small\emph{Department of Mathematics, Universitas Indonesia, 16424 Depok, Indonesia}}
\affil[$\ast$]{Corresponding author. Email: \href{mailto:aldiladipo@sci.ui.ac.id}{aldiladipo@sci.ui.ac.id}}
\maketitle
\begin{abstract}
{\textbf{Abstract}}: This study focuses on analyzing a deterministic SIR model governing the dynamics of the hosts and vectors on an urban network. Our analysis scrutinizes the typical existence--stability of the equilibria as well as the sensitivity of the basic reproductive number. The latter is a groundbreaking finding that has a strong applicative relevance, particularly for addressing the problem of the distribution of job opportunities toward diminishing the disease persistence on the entire network.
~\\
~\\
{\textbf{Keywords}}: \textsf{network--based SIR model, basic reproductive number, existence--stability analysis, sensitivity analysis, decentralized job opportunities}
\vspace*{-5pt}
\end{abstract}


\section{Introduction}
In this study, we examine an urban network consisting of $n$ zones in which a vector--borne disease overwhelms the settling human population. The point of departure in the modeling lies in setting up such network as a weighted graph $(V,E,\XB)$. The set $V\subset\R^2$ denotes a set of zones in which a pair of zones $v_i,v_j\in V$ are connected through directed edges $e_{ji},e_{ij}\in E$ that do not need to coincide. Each zone $v_i\in V$ is associated with zonal population $X_i\in\XB$, which varies over time during the observation. Human mobility can be parameterized if personal traits and locations are disregarded, i.e., it can be considered as the average timely probability of a person in zone $v_j$ to move to another zone $v_i$, denoted as $c_{ij}$. Thus, $c_{ij}X_j$ would yield the number of people who move from $v_j$ to $v_i$ instantaneously. For the sake of our model, we assume that $c_{ii}=0$ and matrix $(c_{ij})$ is column stochastic, i.e., $\sum_{i}c_{ij}=1$ for every $j$. A matrix with such a property can be easily verified by examining either the commuter data or utilizing an approximative strategy, such as the radiation model \cite{SGM2012,MSJ2013}. It is to be noted that the radiation model itself can produce different results depending on how we measure the distance, which can fall between the maximum topological distance, (typical) distance between the centers, or length of each edge $|e_{ij}|$ that connects two centers, provided the edges are pre--defined \cite{WAR2017}. Owing to the notation $\XB=(X_1,\cdots,X_n)$, the timely \emph{flux balance} in zone $v_i$, denoted as $\Delta X_i$, is expressed as
\begin{equation*}
   \Delta X_i =\Phi \XB\quad\text{where }\Phi=(\phi_{ij}) := \begin{dcases}
		c_{ij} & \text{for $i\neq j$} \\
		- \sum_{k=1}^n c_{ki} & \text{for $i=j$}
	\end{dcases}.
\end{equation*}%
The above equation yields the difference between the number of people from all the zones who enter $v_i$ and the number of people in $v_i$ who exit to move into other zones. The preceding assumption applied to the matrix $(c_{ij})$ yields the simplification that $\phi_{ii}=-1$ for all $i$. For the several important properties of the flux matrix $\Phi$, we refer the reader to~\cite{WAR2017}. Our current model is developed from~\cite{WAR2017} for which previously a time scale separation analysis was performed as in~\cite{RAS2013} to exclude the vector dynamics
\begin{subequations}
\label{eq:model}
\begin{align}
\dot{\SB} &= \mu(\IB+\RB)-\diag(\beta)\diag(\IB+\diag(\nu) \NB)^{-1}\diag(\SB)\IB + \kappa \RB + \sigma \rhd \Phi \SB,\label{eq:S}\\
\dot{\IB} &= \diag(\beta)\diag(\IB+\diag(\nu) \NB)^{-1}\diag(\SB)\IB - (\gamma+\mu) \IB + p\sigma\rhd \Phi \IB,\label{eq:I}\\
\dot{\RB} &= \gamma \IB - (\mu+\kappa)\RB + \sigma \rhd \Phi \RB.\label{eq:R}
\end{align}
\end{subequations}%
Note that $\SB=(S_1,\cdots,S_n)$, $\IB=(I_1,\cdots,I_n)$, $\RB=(R_1,\cdots,R_n)$, and $\NB=(N_1,\cdots,N_n)$, denote the susceptible subpopulations, the infective subpopulations, the recovered subpopulations, and the zonal populations ($N_i=S_i+I_i+R_i$) of all the zones, respectively. The involved parameters $\mu,\beta,\nu,\kappa,\gamma,p,\sigma$ denote the death rate, the infection rates, the indicators of the existing infective vectors, the loss--of--immunity rate, the recovery rate, the free--to--bed--rest ratio, and the maximum flux loads, respectively. Mathematically, this SIR model resembles that in~\cite{WAR2017}, except that the infection rates $\beta=(\beta_1,\cdots,\beta_n)$, the maximum flux loads $\sigma=(\sigma_1,\cdots,\sigma_n)$, and the indicators $\nu=(\nu_1,\cdots,\nu_n)$ are not necessarily uniform across the zones. In view of $\Phi=(\Phi_1|\cdots|\Phi_n)$, $\sigma\rhd\Phi$ represents $(\sigma_1\Phi_1|\cdots|\sigma_n\Phi_n)$, which is similar to $(\diag(\sigma)\Phi')'$. Our motivation for this development started from the consideration that some zones can function to a high extent of leisure resorts that attract people who are willing to be outdoor, where the vectors are the most probable to fly. Accordingly, the endemicity of each zone depends on the probability of people to lie around outdoor and during day time, which satisfies the preference of, for example, \emph{Aedes aegypti} mosquitoes as day biters. Another motivation was related to our finding regarding field commuter data in the city of Jakarta, Indonesia \cite{BPS2015}, according to which the average zonal share of people who commute from one zone to all the other zones is unique. In this model, the shares are represented as $\sigma_1,\cdots,\sigma_n$. In the remainder of this section, we elicit some facts regarding the matrix $\sigma\rhd\Phi$ that are essential for further analysis. For pronouncing practical relevance, we assume that $\Phi$ can be expressed as $\Phi=\Phi_+-\id$, where $\Phi_+$ is irreducible.
\begin{proposition}\label{prop:phi}
The following assertions hold true:
\begin{enumerate}[label=\normalfont (A\arabic*)]
\item$\dt(\sigma\rhd\Phi)=0$.\label{ad:detzero}
\item For any real positive $r$, $\dt(\sigma\rhd\Phi-r\id)\neq 0$, and $\sigma\rhd\Phi-r\id$ is a stable matrix, having eigenvalues with negative real part.\label{ad:det}
\item There exists a unique positive solution $x$ $($up to positive scaling$)$ of $\sigma\rhd\Phi x=0$.\label{ad:possol}
\end{enumerate}
\end{proposition}%

\section{Analysis of the equilibria}
We subdivide this section into several parts which we subsequently present for the benefit of the reader. The first part is concerned with the biological relevance of this model. This includes the following results whose proofs are not provided because of their triviality.
\begin{lemma}\label{lem:biology}
The following assertions hold true:
\begin{enumerate}[label=\normalfont (B\arabic*)]
\item The nonnegative orthant $\R^{3n}_+$ is positively invariant under the solution of the SIR model.\label{ad:orthant}
\item By defining all the force terms $S_iI_i\slash (I_i+\nu_i N_i)$ to be zero when $S_i=I_i=R_i=0$, the model admits a unique solution for a given nonnegative initial condition.
\item For a constant total population $\mathcal{N}$, the set~$\Omega:=\{(\SB,\IB,\RB)\in\mathbb{R}^{3n}_+:\,\left\langle e, \SB+ \IB+ \RB\right\rangle=\mathcal{N}\}$ is also positively invariant under the solution of the model.
\end{enumerate}
\end{lemma}

The second part is aimed at determining the existence of a disease--free equilibrium. Apparently, certifying this does not require any further condition.
\begin{theorem}
There always exists a unique disease--free equilibrium $\DFE=\left(\SB,0,0\right)$ with a positive $\SB$ satisfying $\langle e,\SB\rangle=\mathcal{N}$.
\end{theorem}%
\begin{proof}
Setting $\IB=0$, the equilibrium equation originating from \eqref{eq:R} yields $(\sigma\rhd\Phi-(\mu+\kappa)\id)\RB=0$, returning $\RB=0$ according to Proposition~\ref{prop:phi} ad~\ref{ad:det}. Given that $\IB=\RB=0$, the equilibrium equation from~\eqref{eq:S} yields $\sigma\rhd\Phi \SB=0$ that allows a unique positive solution $\SB$ satisfying $\langle e,\SB\rangle =\mathcal{N}$, i.e., according to Proposition~\ref{prop:phi} ad~\ref{ad:possol}.
\end{proof}

The third part discusses the stability of \DFE, which can typically be related to the basic reproductive number. We first recall the mathematical definition of the basic reproductive number according to van den Driessche and Watmough \cite{DW2002}. Rewriting $\left.\nabla_{\IB}\dot{\IB}\right|_{\DFE}:=F-V$, where $F:=\diag(\beta)\diag(\nu)^{-1} +p\sigma \rhd \Phi_+$ and $V:=(\gamma+\mu) \id+p\diag(\sigma)$, we obtain the basic reproductive number as
\begin{equation*}
\mathcal{R}_0:=\rho\left(\underbrace{FV^{-1}}_{:=G}\right)\quad\text{where}\quad G=(g_{ij})=\begin{dcases}
\frac{p\sigma_j\phi_{ij}}{\gamma+\mu+p\sigma_j} & \text{for }i\neq j\\
\frac{\beta_i\slash \nu_i}{\gamma+\mu+p\sigma_j} & \text{for }i= j
\end{dcases}.
\end{equation*}%
Note that irreducibility of $\Phi_+$ implies that of \emph{next generation matrix} $G$. Regarding the local stability, it is clear that the eigenvalues of the Jacobian of our SIR model evaluated at \DFE~are those of the submatrices $\sigma \rhd\Phi$, $\diag(\beta)\diag(\nu)^{-1}-(\gamma+\mu)\id +p\sigma\rhd\Phi$, and $-(\gamma+\kappa)\id + \sigma \rhd \Phi$. Proposition~\ref{prop:phi} ad~\ref{ad:detzero} results in the existence of zero eigenvalues of $\sigma \rhd\Phi$, which further leads to the center manifold theory; this extension is omitted here because of some technical difficulties. Concurrently, the invariance principle of LaSalle provides insight on the global stability of \DFE~in~$\Omega$ in the sense of Lyapunov. 
\begin{theorem}\label{thm:stableDFE}
If $p\max_i\sigma_i\leq \gamma+\mu$ and $\mathcal{R}_0< 1$, then \DFE~is globally asymptotically stable in $\Omega$.
\end{theorem}
\begin{proof}
Working on the following continuously differentiable function $V:=\langle \IB,\IB\rangle$, we found that $\dot{V}= 2\langle \IB,\dot{\IB}\rangle=2\langle \IB,\diag(\beta)\diag(\IB+\diag(\nu) \NB)^{-1}\diag(\SB)\IB - (\gamma+\mu) \IB + p\sigma\rhd\Phi \IB\rangle\leq 2\langle \IB, (F-V)\IB\rangle\leq 0$ because $\SB,\IB,\RB\geq 0$, where $\SB\leq \NB$ implies that $S_i\slash (I_i+\nu_i N_i)\leq N_i\slash (I_i+\nu_i N_i)\leq 1\slash \nu_i$ and $\RM_0<1$ if and only if $F-V$ has eigenvalues with negative real part. Note that $\dot{V}=0$ if and only if $\IB=0$ or $\dot{\IB}=0$. We show that $\dot{\IB}=0$ implies $\IB=0$ in $\Omega$, whereas its converse is trivial. Assume that $\dot{\IB}=0$, but $\IB\neq 0$. Then, the equilibrium equation originating from \eqref{eq:I} yields
\begin{align*}
0&=\dt\left(\diag(\beta)\diag(\IB+\diag(\nu)\NB)^{-1}\diag(\SB)-(\gamma+\mu)\id+p\sigma\rhd \Phi\right)\\
&=\dt\left(\diag(\beta)\diag(\SB)-(\gamma+\mu)\diag(\IB+\diag(\nu)\NB)+p\sigma\rhd\Phi\diag(\IB+\diag(\nu)\NB)\right).
\end{align*}%
The specification $\RM_0<1$ yields via the Gerschgorin analysis over $F-V$ that $|\beta_i-(\gamma+\mu)\nu_i-p\sigma_i\nu_i|>p\sigma_i\nu_i$, and based on $p\max_i\sigma_i\leq \gamma+\mu$, we found that $\gamma+\mu\geq p\sigma_i$ for all $i$. Consequently, the last matrix achieves zero determinant if and only if there exists an index $j$ such that $S_j=I_j=R_j=0$. By $\dot{\IB}=0$, we obtain $0=\dot{I}_j=\sum_{j\neq i}\phi_{ij}I_j$. However, this cannot be the case in $\Omega$ because $\IB\neq 0$, leading to a contradiction. Accordingly, the largest invariant set contained in $\{(\SB,\IB,\RB)\in\Omega:\dot{V}=0\}$, toward which any trajectory starting in $\Omega$ converges, is $\{\DFE\}$.
\end{proof}

The fourth part is devoted to determining the existence of an endemic equilibrium from our model. In view of the complexities after the classical substitutions, we used the concept of ensuring the existence of a nontrivial equilibrium without even knowing the explicit formulation of the basic reproductive number \cite{WSS2017}, which we present in Lemma~\ref{lem:existnontrivial}. Based on the corresponding result, the recovered subpopulations of the endemic equilibrium can now be characterized by
\begin{equation*}
\RB=\gamma((\mu+\kappa)\id-\sigma\rhd \Phi)^{-1}\IB.
\end{equation*}%
We derive some sufficient and necessary conditions for which $((\mu+\kappa)\id-\sigma\rhd\Phi)^{-1}\geq 0$ to ensure $\RB\geq0$, owing to the aid of \cite{Hor1997}.
\begin{proposition}\label{prop:neccR}
If $\mu+\kappa>\max_i\sigma_i-\min_i\sigma_i$, then $\RB\geq 0$. If $\RB\geq 0$ and denoting $A:=\diag(\max_i\sigma_ie-\sigma)+\sigma\rhd\Phi_+$ where $\rk(A)\geq 2$, then
\begin{equation*}
\mu+\kappa+\max_i\sigma_i>
\begin{dcases}
\frac{|\tr(A)|}{\rk(A)}+\sqrt{\frac{\tr(A^2)-\tr^2(A)\slash \rk(A)}{\rk(A)(\rk(A)-1)}} & \text{if }\tr(A^2)\geq \tr^2(A)\slash \rk(A)\\
\sqrt{\frac{\tr^2(A)\slash \rk(A)-\tr(A^2)}{\rk(A)(\rk(A)-1)}} & \text{if }\tr^2(A)\slash \rk(A)\geq \tr(A^2)
\end{dcases}
\end{equation*}
where $\tr,\rk$ denote the trace and rank of their argument, respectively.
\end{proposition}
Now we can discuss the result regarding the existence of an endemic equilibrium $\EE$. Note that the result does not ensure uniqueness, for which we consider an existing sample and assume its uniqueness for further stability analysis. 
\begin{lemma}\label{lem:existnontrivial}
If $\mu+\kappa>\max_i\sigma_i-\min_i\sigma_i$ and $\RM_0>1$ for which $\RM_0$ is sufficiently close to $1$, then a nonzero nonnegative endemic equilibrium $\EE$ exists.
\end{lemma}
\begin{proof}
We are attempting to transform the equilibrium equations from all the subpopulations into a canonical equation as per~\cite{WSS2017}. Denoting $A:=\gamma((\mu+\kappa)\id-\sigma\rhd\Phi)^{-1}$, we obtain from \eqref{eq:I} and \eqref{eq:R} in conjunction with Proposition~\ref{prop:neccR}
\begin{align*}
0&= \diag(\beta)\diag(\diag(\nu) \SB+\underbrace{(\id+\diag(\nu)+\diag(\nu) A)}_{:=B}\IB )^{-1}\diag(\SB)\IB - (\gamma+\mu) \IB + p\sigma\rhd\Phi \IB\\
&=\diag(\beta)\diag(\nu)^{-1}\left(\sum_{j=0}^{\infty} (-1)^j\diag (B\IB)^{j}\diag(\nu \SB)^{-j}\right)\IB - (\gamma+\mu) \IB + p\sigma\rhd\Phi \IB
\end{align*}%
We note that $B$ is bounded while the solution state converges to $\DFE$ as $\IB\rightarrow 0$, implying
\begin{align*}
0&= \diag(\beta)\diag(\nu)^{-1}\left(\id+\mathcal{O}(\lVert \IB\rVert)\right)\IB - (\gamma+\mu) \IB + p\sigma\rhd\Phi \IB\quad \text{as }\IB\rightarrow 0.
\end{align*}%
Moreover, the ansatz $\JB:=((\gamma+\mu)\id + p\diag(\sigma))\IB$ apparently follows the aforementioned canonical equation
\begin{align*}
0&= \left(\id - FV^{-1}\right)\JB + \mathcal{O}(\lVert \JB\rVert^2).
\end{align*}%
With $\RM_0=\rho(FV^{-1})>1$ being close to $1$ and $FV^{-1}$ being irreducible, there exists a positive $\JB$ that solves the last canonical equation. This immediately suggests the existence of a positive $\IB$ solving the equilibrium equations by the positive diagonal entries of $(\gamma+\mu)\id + p\diag(\sigma)$. We next analyze the equilibrium equations originating from \eqref{eq:S}, where $\SB=0$ implies $\mu(\IB+\RB)+\kappa \RB=0$; however, this equality cannot be true because $\IB$ is positive. Further inspection to \eqref{eq:S} shows that  $aS_j^2+bS_j+c=0$ where
\begin{align*}
a&=p\sigma_j\nu_j\\
b&=((1+\nu_j)I_j+\nu_j R_j)p\sigma_j+\beta_j I_j-\nu_j\left(\mu(I_j+R_j)+\kappa R_j+p\sum_{k\neq j}\sigma_k\phi_{jk}S_k\right)\\
c&=-\left(\mu(I_j+R_j) + \kappa R_j+p\sum_{k\neq j}\sigma_k\phi_{jk}S_k\right)((1+\nu_j)I_j+\nu_j R_j)
\end{align*}%
holds for all $j$. The claim that all the two roots $S_j<0$ for some $j$ cannot be the case, since it would produce $c>0$ where some more $S_k<0$, $k\neq j$. The remaining possibilities are that $S_j$ are complex conjugate with positive real part for some $j$ while $S_k<0$ for some $k\neq j$ (i.e. $c>0$ and apparently $b>0$) which we discard since it can only happen under certain structures of $\sigma,\phi$ and the cardinality of $j$ must be even (i.e. real $S_k$ would require its $c$ to be real, whereas $c$ is in terms of a sum of the complex $S_j$), and that all $S_j$ are nonnegative but $\SB\neq 0$ (i.e. $c<0$) which we point out as with its validity for any structures of $\sigma,\phi$ to justify this lemma.
\end{proof}%

The last part provides the result regarding the stability of $\EE$. The corresponding result is given in Theorem~\ref{thm:stablenontrivial} whose proof urges the following intermediate result.
\begin{lemma}\label{lem:interm}
If $X,Y\in\R^n$ and $X\geq 0$, then it holds $X\cdot Y\geq \lVert X\rVert_{\infty}e\cdot Y$.
\end{lemma}

\begin{theorem}\label{thm:stablenontrivial}
If $\mu+\kappa>\max_i\sigma_i-\min_i\sigma_i$, $\RM_0>1$ for which $\RM_0$ is sufficiently close to $1$, and the existing endemic equilibrium $\EE=(\SB^{\ast},\IB^{\ast},\RB^{\ast})$~is unique, then it is globally asymptotically stable in $\Omega$.
\end{theorem}
\begin{proof}
We are considering the following continuously differentiable function
\begin{align*}
V&:=\langle e,\SB-\SB^{\ast}\rangle - \langle \SB^{\ast},\log(\SB)-\log(\SB^{\ast})\rangle+ \langle e,\IB-\IB^{\ast}\rangle - \langle \IB^{\ast},\log(\IB)-\log(\IB^{\ast})\rangle\\
&+ \langle e,\RB-\RB^{\ast}\rangle - \langle \RB^{\ast},\log(\RB)-\log(\RB^{\ast})\rangle,
\end{align*}%
where $\log(\SB) = (\log S_1,\cdots, \log S_n)$. This function is always nonnegative in $\Omega$ and vanishes at \EE. This formula does not change any meaning as to an equilibrium state $X_j^{\ast}$ is zero, since then the limiting formula for such state simply returns $X_j-X_j^{\ast}-X_j^{\ast}\log(X_j\slash X_j^{\ast})\rightarrow X_j$. In \cite{WAR2017}, it was shown that $\langle \Phi \SB,\diag(\SB)^{-1}\SB^{\ast}\rangle\geq 0$, which is also valid for the infective and recovered subpopulations. Note that the last estimate bears a strict inequality if at least one equilibrium state is zero. Through a constant total population, we obtain
\begin{align*}
\dot{V}&= \langle \dot{\SB},e-\diag(\SB)^{-1}\SB^{\ast}\rangle + \langle \dot{\IB},e-\diag(\IB)^{-1}\IB^{\ast}\rangle + \langle \dot{\RB},e-\diag(\RB)^{-1}\RB^{\ast}\rangle\text{ or }\\
-\dot{V}&= \langle \dot{\SB},\diag(\SB)^{-1}\SB^{\ast}\rangle + \langle \dot{\IB},\diag(\IB)^{-1}\IB^{\ast}\rangle + \langle \dot{\RB},\diag(\RB)^{-1}\RB^{\ast}\rangle\\
&\geq \langle \mu(\IB+\RB) - \diag(\beta)\diag(\IB+\diag(\nu) \NB)^{-1}\diag(\IB)\SB+\kappa\RB,\diag(\SB)^{-1}\SB^{\ast}\rangle \\
&+ \langle \diag(\beta)\diag(\IB+\diag(\nu) \NB)^{-1}\diag(\SB)\IB-(\gamma+\mu)\IB,\diag(\IB)^{-1}\IB^{\ast}\rangle\\
& + \langle \gamma\IB-(\mu+\kappa)\RB,\diag(\RB)^{-1}\RB^{\ast}\rangle\\
&=\left\langle e,\mu\left(\diag(\IB)\diag(\SB)^{-1}\SB^{\ast}- \IB^{\ast}\right)+ \frac{\diag(\beta)}{2}\diag(\IB+\diag(\nu)\NB)^{-1}\left(\diag(\SB)\IB^{\ast}-\diag(\SB^{\ast})\IB\right)\right\rangle\\
&+\left\langle e,(\mu+\kappa)(\diag(\RB)\diag(\SB)^{-1}\SB^{\ast}-\RB^{\ast})+\gamma (\diag(\IB)\diag(\RB)^{-1}\RB^{\ast}-\IB^{\ast})\right\rangle\\
&+\left\langle e,\frac{\diag(\beta)}{2}\diag(\IB+\diag(\nu)\NB)^{-1}\left(\diag(\SB)\IB^{\ast}-\diag(\SB^{\ast})\IB\right)\right\rangle\\
&\geq \max\left\{\left\langle e,\mu\IB^{\ast}\right\rangle,\left\langle e,\frac{\diag(\beta)}{2}\diag(\IB+\diag(\nu)\NB)^{-1}\diag(\SB^{\ast})\IB\right\rangle\right\} \\
&\left\langle e,\diag(\IB)\diag(\IB^{\ast})^{-1}\diag(\SB)^{-1}\SB^{\ast}+\diag(\IB)^{-1}\diag(\IB^{\ast})\diag(\SB^{\ast})^{-1}\SB-2e\right\rangle\\
&+\max\left\{\left\langle e,(\mu+\kappa)\RB^{\ast}\right\rangle, \left\langle e,\gamma\IB^{\ast}\right\rangle, \left\langle e,\frac{\diag(\beta)}{2}\diag(\IB+\diag(\nu)\NB)^{-1}\diag(\SB^{\ast})\IB\right\rangle\right\}\\
&\left\langle e,\diag(\RB)\diag(\RB^{\ast})^{-1}\diag(\SB)^{-1}\SB^{\ast}+\diag(\IB)\diag(\IB^{\ast})^{-1}\diag(\RB)^{-1}\RB^{\ast}\right.\\
&\left.+\diag(\IB)^{-1}\diag(\IB^{\ast})\diag(\SB^{\ast})^{-1}\SB-3e\right\rangle\\
&\geq 0
\end{align*}%
by Lemma~\ref{lem:interm} and the arithmetic--geometric mean inequality. On applying the uniqueness assumption to $\EE$, the largest invariant set contained in $\{(\SB,\IB,\RB)\in\Omega:\dot{V}=0\}$ is $\{\EE\}$.
\end{proof}

\section{Sensitivity of the basic reproductive number}
The next generation matrix $G$ is essentially nonnegative, i.e., it has nonnegative off--diagonal entries. Moreover, the matrix $\RM_0\id-G$ is a singular M--matrix. The sensitivity of the basic reproductive number with respect to the flux matrix $\Phi=(\phi_{sj})$ is measured by the following derivative
\begin{equation*}
\partial_{\phi_{sj}}\RM_0=-\sum_{i\neq j}\partial_{g_{ij}}\RM_0\frac{p^2\sigma_j^2\phi_{ij}}{\left(\gamma+\mu+p\sigma_j\right)^2}-\partial_{g_{jj}}\RM_0\frac{p\sigma_j\beta_j\slash\nu_j}{\left(\gamma+\mu+p\sigma_j\right)^2}.
\end{equation*}
Having known that the sensitivity strongly depends on $\partial_{g_{ij}}\RM_0$, we briefly reviewed some of the results regarding the sensitivity of the Perron vector of an essentially nonnegative matrix. Clearly, $\RM_0$ is directly proportional to $\beta_j$ and $\sigma_j$, but inversely proportional to $\nu_j$ at all $j$ owing to the lower bounds cf. Proposition~\ref{prop:neccR} and the $1$--norm as the upper bound of the spectral radius. The sensitivity matrix $(\partial_{g_{ij}}\RM_0)$ is $\id-(\RM_0\id-G)(\RM_0\id-G)^{\#}=vw>0$, where $Q^{\#}$ denotes the group inverse of $Q$ (see~\cite{BG1973} for the details) and $w,v$ denotes the normalized left and right eigenvectors of $G$ associated with $\RM_0$ such that $wv=1$ \cite{MS1978}. Moreover, $\partial_{g_{ii}}\RM_0<1$, and $\RM_0$ is a convex function of all diagonal elements $g_{ii}$ \cite{Coh1978}. If $G$ can be designed in such a way that it serves as a column stochastic matrix, then $\partial_{g_{1j}}\RM_0=\cdots=\partial_{g_{nj}}\RM_0$ for all $j$ where $\sum_{j}\partial_{g_{ij}}\RM_0=1$ for all $i$, and certainly $0<\partial_{g_{ij}}\RM_0<1$ \cite{DN1984}. Additionally, if $G$ is symmetric, by which $\Phi_+$ is symmetric and each column sum is equal to its $1$--norm, then $\RM_0$ is a convex function for all $g_{ij}$, keeping in mind the positivity of the first derivatives \cite{Lax1958}. These findings suggest that $\RM_0$ decreases as any $\phi_{sj}$ increases, whereas the other entries are kept constant. Moreover, the higher $\sigma_j$, the faster $\RM_0$ decreases with respect to any off--diagonal entry in the $j$th column of $\Phi$.   

\section{Conclusions}
We have analyzed a network--based SIR model by examining the existence--stability of a disease--free equilibrium and endemic equilibrium. The analysis highlights the importance of the maximum flux loads in all the observed zones, taking into account that further numerical realization may fail to produce both the equilibria, if sufficient conditions relating the maximum flux loads for the existence are violated. Furthermore, the result regarding the sensitivity of the basic reproductive number with respect to the flux matrix yields two final inferences. First, the disease persistence can be diminished as more jobs are invested in zones where people usually do not go for work, while simultaneously investing in the adequate transportation system to make these jobs accessible. Second, increasing the number of jobs in the zones in which a large number of people move out for work would also help diminish the persistence of the disease.

\section*{Acknowledgements}
This work has been supported financially by Indonesia Ministry of Research and Higher Education (Kemenristekdikti RI), under the research project PUPT research grant scheme 2017,\\ No. 2701/UN2.R3.1/HKP05.00/2017. We thank the referees for their useful suggestions on the earlier version of the paper.


\begin{thebibliography}{90}

\bibitem{SGM2012} F. Simini, M. C. Gonz\'{a}lez, A. Maritan, and A-L Barab\'{a}si. A universal model for mobility and migration patterns. \emph{Nature} 484:96--100 (2012)

\bibitem{MSJ2013} A. P. Masucci, J. Serras, A. Johansson, and M. Batty. Gravity versus radiation models: On the importance of scale and heterogeneity in commuting flows. \emph{Physical Review E} 88:022812--8 (2013)

\bibitem{WAR2017} K. P. Wijaya, D. Aldila, and M. K. Ramadhan. Modeling and controling the spread of a vector-borne disease on an urban network. Arxiv preprint (2017)

\bibitem{RAS2013} F. Rocha, M. Aguiar, M. Souza, and N. Stollenwerk. Time-scale separation and centre manifold analysis describing vector-borne disease dynamics. \emph{International Journal of Computer Mathematics} 90(10):2105--2125 (2013)

\bibitem{BPS2015} Berita Resmi Statistik: Komuter DKI Jakarta Tahun 2014. (Statistics Indonesia, Jakarta, No.12/02/31/Th.XVII, 2015)

\bibitem{DW2002} P. van den Driessche and J. Watmough. Reproduction numbers and sub-threshold endemic equilibria for compartmental models of disease transmission. \emph{Mathematical Biosciences} 180(1):29--48 (2002)

\bibitem{WSS2017} K. P. Wijaya, Sutimin, E. Soewono, and T. G\"{o}tz. On the existence of a nontrivial equilibrium in relation to the basic reproductive number. \emph{International Journal of Applied Mathematics and Computer Science} 3(27):623--636 (2017)

\bibitem{Hor1997} B. G. Horne. Lower bounds for the spectral radius of a matrix. \emph{Linear Algebra and Its Applications}, 263:261--273 (1997)

\bibitem{BG1973} A. Ben--Israel and T. N. Greville. \emph{Generalized Inverses: Theory and Applications}. (Academic Press, New York, 1973)

\bibitem{MS1978} C. D. Mayer, Jr. and M. W. Stadelmaier. Singular M--matrices and inverse positivity. \emph{Linear Algebra and Applications} 22:139--156 (1978)

\bibitem{Coh1978} J. E. Cohen. Derivatives of the spectral radius as a function of nonnegative matrix elements. \emph{Mathematical Proceedings of the Cambridge Philosophical Society} 83(2):183--190 (1978) 

\bibitem{DN1984} E. Deutsch and M. Neumann. Derivatives of the Perron root at an essentially nonnegative matrix and the group inverse of an M-matrix. \emph{Journal of Mathematical Analysis and Applications} 102(1):1--29 (1984)

\bibitem{Lax1958} P. D. Lax. Differential equations, difference equations and matrix theory. \emph{Communication on Pure and Applied Mathematics} 11:175--194 (1958)

%
%
%
%

\end{thebibliography}
\end{document}